\documentclass{article}

\usepackage[centertags]{amsmath}
\usepackage{hyperref}
\usepackage{amsfonts}
\usepackage{amssymb}
\usepackage{amsthm}
\usepackage{newlfont}
\usepackage{amscd}
\usepackage{amsmath,amscd}
\usepackage{graphicx}
\usepackage[all]{xy}
\usepackage{verbatim}

\usepackage{color}

\newcommand{\adj}{\rightleftarrows}
\newcommand{\ten}{\otimes}

\newcommand{\cA}{\mathcal{A}}
\newcommand{\cB}{\mathcal{B}}
\newcommand{\cC}{\mathcal{C}}
\newcommand{\cD}{\mathcal{D}}

\newcommand{\cF}{\mathcal{F}}

\newcommand{\cM}{\mathcal{M}}

\newcommand{\cS}{\mathcal{S}}

\newcommand{\cW}{\mathcal{W}}

\newtheorem{thm}{Theorem}[section]

\newtheorem{cor}[thm]{Corollary}

\newtheorem{lem}[thm]{Lemma}
\newtheorem{prop}[thm]{Proposition}

\theoremstyle{definition}
\newtheorem{define}[thm]{Definition}

\theoremstyle{remark}
\newtheorem{rem}[thm]{Remark}

\def\Csep{\mathtt{SC^*}}
\def\CCsep{\mathtt{CSUC^*}}
\def\R{\textup{R}}
\def\K{\textup{K}}

\DeclareMathOperator{\Top}{Top}
\DeclareMathOperator{\Sing}{Sing}

\DeclareMathOperator{\Pro}{Pro}
\DeclareMathOperator{\Ind}{Ind}
\DeclareMathOperator{\Map}{Map}
\DeclareMathOperator{\id}{id}

\DeclareMathOperator{\Hom}{Hom}
\DeclareMathOperator{\Mor}{Mor}

\DeclareMathOperator{\precolim}{colim}
\DeclareMathOperator{\op}{op}
\DeclareMathOperator{\CM}{CM}
\DeclareMathOperator{\Lw}{Lw}

\DeclareMathOperator{\coSp}{coSp}
\def\colim{\mathop{\precolim}}

\def \mcal{\mathcal}

\DeclareFontEncoding{OT2}{}{} 
\DeclareTextFontCommand{\textcyr}{\fontencoding{OT2}\fontfamily{wncyr}\fontseries{m}\fontshape{n}\selectfont}

\begin{document}
\title{The two out of three property in ind-categories and a convenient model category of spaces}

\author{Ilan Barnea}

\maketitle

\begin{abstract}
In \cite{BaSc2}, the author and Tomer Schlank introduced a much weaker homotopical structure than a model category, which we called a ``weak cofibration category". We further showed that a small weak cofibration category induces in a natural way a model category structure on its ind-category, provided the ind-category satisfies a certain two out of three property. The main purpose of this paper is to give sufficient intrinsic conditions on a weak cofibration category for this two out of three property to hold. We consider an application to the category of compact metrizable spaces, and thus obtain a model structure on its ind-category. This model structure is defined on a category that is closely related to a category of topological spaces and has many convenient formal properties. A more general application of these results, to the (opposite) category of separable $C^*$-algebras, appears in a paper by the author, Michael Joachim and Snigdhayan Mahanta \cite{BJM}.
\end{abstract}

\tableofcontents

\section{Introduction}
In \cite{BaSc1} the author and Tomer Schlank introduced the concept of a weak fibration category. This is a category $\cC$, equipped with two subcategories of weak equivalence and fibrations, satisfying certain axioms. This notion is closely related to the notion of a category of fibrant objects due to Brown \cite{Bro}. The only difference is that we require arbitrary pullbacks to exist, and we do not require every object to be fibrant. A weak fibration category is a much weaker notion than a model category and its axioms are much easily verified.

If $\cC$ is any category, its pro-category $\Pro(\cC)$ is the category of inverse systems in $\cC$. That is, objects in $\Pro(\cC)$ are diagrams $I\to\cC$, with $I$ a cofiltered category. If $X$ and $Y$ are objects in $\Pro(\cC)$ having the same indexing category, then a natural transformation $X\to Y$ defines a morphism in $\Pro(\cC)$, but morphisms in $\Pro(\cC)$ are generally more flexible.

Given a weak fibration category $\cC$, there is a very natural way to induce a notion of weak equivalence on the pro-category $\Pro(\cC)$. Namely, we define the weak equivalences in $\Pro(\cC)$ to be the smallest class of maps that contains all (natural transformations that are) levelwise weak equivalences, and is closed under isomorphisms of maps in $\Pro(\cC)$. If $\cW$ is the class of weak equivalences in $\cC$, then we denote the class of weak equivalences in $\Pro(\cC)$ by $\Lw^{\cong}(\cW)$. The maps in $\Lw^{\cong}(\cW)$ are called \emph{essentially levelwise weak equivalences} by Isaksen \cite{Isa}. Note, however, that $\Lw^{\cong}(\cW)$ may not satisfy the two out of three property. Weak fibration categories for which $\Lw^{\cong}(\cW)$ satisfies the two out of three property are called pro-admissible.

The main result in \cite{BaSc1} is that a pro-admissible weak fibration category $\cC$ induces in a natural way a model structure on $\Pro(\cC)$, provided $\cC$ has colimits and satisfies a technical condition called \emph{homotopically small}. In \cite{BaSc2}, we explain that an easy consequence of this result is that any small pro-admissible weak fibration category $\cC$ induces a model structure on $\Pro(\cC)$.

Dually, one can define the notion of a weak \emph{cofibration} category (see Definition \ref{d:weak_fib}), and deduce that a small \emph{ind}-admissible weak cofibration category induces a model structure on its \emph{ind}-category (which is the dual notion of a pro-category). This is the setting in which we work with in this paper, however, everything we do throughout the paper is completely dualizable, so it can also be written in the ``pro" picture.

Since the axioms of a weak cofibration category are usually not so hard to verify, the results above essentially reduce the construction of model structures on ind-categories of small categories to the verification of the two out of three property for $\Lw^{\cong}(\cW)$. This is usually not an easy task.

The main purpose of this paper is to prove theorems giving sufficient intrinsic conditions on a weak cofibration category from which one can deduce its ind-admissibility. We apply these results to the category of compact metrizable spaces. This category has a natural weak cofibration structure, which we show is ind-admissible. We do not know to deduce the ind-admissibility of this weak cofibration category using the methods given in \cite{BaSc2} (see Remark \ref{r:bs2}).

We will now state our main result. For this, we first need a definition:
\begin{define}
Let $\cC$ a weak cofibration category. An object $D$ in $\cC$ is called:
\begin{enumerate}
\item Cofibrant, if the unique map $\phi\to D$ from the initial object is a cofibration.
\item Fibrant, if for every acyclic cofibration $A\to B$ in $\cM$ and every diagram of the form
$$
\xymatrix{ A \ar[d]\ar[r] & D\\
 B & }
$$
there is a lift $B\to D$.
\end{enumerate}
\end{define}

We can now formulate our main result in this paper which is the following criterion for ind-admissibility:
\begin{thm}[see Theorem \ref{t:main2}]\label{t:main2 0}
Let $\cC$ be a small simplicial weak cofibration category (see Definition \ref{d:tensored}). Suppose that the following conditions are satisfied:
\begin{enumerate}
\item $\cC$ has finite limits.
\item Every object in $\cC$ is fibrant and cofibrant.
\item A map in $\cC$ that is a homotopy equivalence in the simplicial category $\cC$ is also a weak equivalence.
\end{enumerate}
Then $\cC$ is ind-admissible.
\end{thm}

In this paper we bring an application to the category of compact metrizable spaces and continuous maps, denoted $\CM$. This category has a natural weak cofibration structure which we now define.
\begin{define}
A map $i: X\to Y$ in $\CM$ is called a \emph{Hurewicz cofibration} if for every $Z\in\CM$ and every diagram of the form
$$
\xymatrix{ \{0\}\times Y\coprod_{\{0\}\times X}[0,1]\times X\ar[d]\ar[r] & Z\\
 [0,1]\times Y & }
$$
there exists a lift $[0,1]\times Y\to  Z$.
\end{define}

\begin{prop}[\ref{p:CM is special}, \ref{p:Hurewicz}]
With the weak equivalences being the homotopy equivalences and the cofibrations being the Hurewicz cofibrations, $\CM$ becomes a weak cofibration category that satisfies the hypothesis of Theorem \ref{t:main2 0}. Thus $\CM$ is ind-admissible.
\end{prop}

\begin{rem}\label{r:bs2}
  In \cite{BaSc2} we also give sufficient intrinsic conditions on a weak cofibration category for its ind-admissibility. Namely, we show in \cite[Corollary 3.8]{BaSc2} that if our weak cofibration category has \emph{proper factorizations} then it is ind-admissible. The author does not know a way to deduce the ind-admissibility of $\CM$ from this result since the author does not know of any factorization in $\CM$ into a weak equivalence followed by a right proper map (see Definition \ref{d:proper} and Proposition \ref{p:compose}).
\end{rem}

Since $\CM$ is ind-admissible, we have an induced model structure on the category $\Ind(\CM)$. We further show in Theorem \ref{t:main4} that:
\begin{thm}\label{t:main4 0}
Let $\cW$ denote the class of weak equivalences and let $\cC$ denote the class of cofibrations in $\CM$. Then there exists a simplicial model category structure on $\Ind(\CM)$ such that:
\begin{enumerate}
\item The weak equivalences are $\mathbf{W} := \Lw^{\cong}(\cW)$.
\item The fibrations are $\mathbf{F} :=(\cC\cap \cW)^{\perp}$.
\end{enumerate}

Moreover, this model category is finitely combinatorial, with set of generating cofibrations $\cC$ and set of generating acyclic cofibrations $\cC\cap \cW$.

The model category $\Ind(\CM)$ has the following further properties:
\begin{enumerate}
\item Every object in $\Ind(\CM)$ is fibrant.
\item $\Ind(\CM)$ is proper.
\item $\Ind(\CM)$ is a cartesian monoidal model category.
\end{enumerate}
\end{thm}

It is well known that the category of topological spaces, with its usual model structure (with weak homotopy equivalences and Serre fibrations), while being a fundamental object in homotopy theory, lacks some convenient properties as a model category. For example, it is not cartesian closed and not combinatorial. Cartesian closure can easily be achieved by restricting to the subcategory of \emph{compactly generated Hausdorff} spaces, however, this category of spaces is still not locally presentable, and thus, not combinatorial. Simplicial sets form a Quillen equivalent combinatorial (and cartesian closed) model category, however, simplicial sets are discrete combinatorial objects and are quite far from the geometric intuition. For some examples arising from geometric situations, it is useful to have a model for topological spaces which is both convenient as a model category, and who's objects have a more topological nature. One attempt to solve this problem is J. H. Smith's idea of \emph{delta-generated spaces} (see \cite{FaRo}).

The model structure that we construct on $\Ind(\CM)$ can be viewed as another possible solution to this problem. First, as we have shown, it has many convenient properties. It is simplicial, cartesian closed, finitely combinatorial, proper and every object in it is fibrant. On the other hand, the underlying category of this model category contains, as a full reflective subcategory, a very large category of spaces (which includes all $CW$-complexes). Furthermore, it can be shown that a colocalization of this model structure is Quillen equivalent to the usual model structure on topological spaces. These last two claims will be addressed in a future paper.

The Gel'fand--Na{\u{\i}}mark correspondence implies that the category of compact metrizable spaces is equivalent to the opposite category of commutative separable unital $C^*$-algebras, which we denote $\CCsep$. Since we have natural equivalences of categories
$$\Ind(\CM)^{\op}\simeq\Pro(\CM^{\op})\simeq\Pro(\CCsep),$$
we see that we obtain a model structure also on the pro-category of commutative separable unital $C^*$-algebras, possessing (the dual of) all the above mentioned properties.

A more general application of our main result in this paper is to the weak fibration category of separable $C^*$-algebras denoted $\Csep$. This application appears in the paper \cite{BJM} by the author, Michael Joachim and Snigdhayan Mahanta. In \cite{BJM}, we use the resulting model structure on $\Pro(\Csep)$ to extend Kasparov's bivariant $\K$-theory category (as well as other bivariant homology theories) from separable $C^*$-algebras to \emph{projective systems} of separable $C^*$-algebras (that is, to objects of $\Pro(\Csep)$).

\subsection{Organization of the paper}

We begin in Section \ref{s:prelim} with a brief account of the necessary background on ind-categories and homotopy theory in ind-categories. We also define the notion of a quasi model category. This notion is weaker then a model category but stronger then the notion of an almost model category, defined in \cite{BaSc2}. In Section \ref{s:left quasi model} we prove results about quasi model categories (generalized from the theory of model categories) that we will need in the next section. In Section \ref{s:Criteria}, we prove our main result of this paper, Theorem \ref{t:main2}, which gives sufficient conditions for the ind-admissibility of a weak cofibration category. We end with Section \ref{s:CM} in which we give an application to the weak cofibration category of compact metrizable spaces.

\subsection{Acknowledgements}
I would like to thank Tomer Schlank, Michael Joachim and Snigdhayan Mahanta for many fruitful conversations. Especially, I would like to thank the last two for the idea of the proof of Proposition \ref{l:right_proper}.

\section{Preliminaries: homotopy theory in ind-categories}\label{s:prelim}
In this section we review the necessary background on ind-categories and homotopy theory in ind-categories. The results presented here are not new, but we recall them for the convenience of the reader. We do bring one new definition, that of a \emph{quasi model category} (see Definition \ref{d:quasi model}). Most of the references that we quote are written for pro-categories, but we bring them here translated to the ``ind" picture which we use in this paper. Standard references on pro-categories include \cite{AM} and \cite{SGA4-I}. For the homotopical parts the reader is referred to \cite{BaSc}, \cite{BaSc1}, \cite{BaSc2}, \cite{BaSc3}, \cite{EH} and \cite{Isa}.

\subsection{Ind-categories}
In this subsection we bring general background on ind-categories.

A non-empty category $I$ is called \emph{filtered}  if the following conditions hold: for every pair of objects $s$ and $t$ in $I$, there exists an object $u\in I$, together with
morphisms $s\to u$ and $t\to u$; and for every pair of morphisms $f$ and $g$ in $I$, with the same
source and target, there exists a morphism $h$ in $I$ such that $h\circ f=h\circ g$. A category is called \emph{small} if it has only a set of objects and a set of morphisms.

If $\mcal{C}$ is any category, the category $\Ind(\mcal{C})$ has as objects all diagrams in $\cC$ of the form $I\to \cC$ such that $I$ is small and filtered. The morphisms are defined by the formula
$$\Hom_{\Ind(\mcal{C})}(X,Y):=\lim_s \colim_t \Hom_{\mcal{C}}(X_s,Y_t).$$
Composition of morphisms is defined in the obvious way.

Thus, if $X:I\to \mcal{C}$ and $Y:J\to \mcal{C}$ are objects in $\Ind(\mcal{C})$, providing a morphism $X\to Y$ means specifying for every $s$ in $I$ an object $t$ in $J$ and a morphism $X_s\to Y_t$ in $\mcal{C}$. These morphisms should satisfy a compatibility condition. In particular, if $p:I\to J$ is a functor, and $\phi:X\to Y\circ p=p^*Y$ is a natural transformation, then the pair $(p,\phi)$ determines a morphism $\nu_{p,\phi}:X\to Y$ in $\Ind(\cC)$ (for every $s$ in $I$ we take the morphism $\phi_s:X_{s}\to Y_{p(s)}$). Taking $X=p^*Y$ and $\phi$ to be the identity natural transformation, we see that any $p:I\to J$ determines a morphism $\nu_{p,Y}:p^*Y\to Y$ in $\Ind(\cC)$. If $I=J$ and we take $p$ to be the identity functor, we see that any natural transformation $X\to Y$ induces a morphism in $\Ind(\cC)$.

The functor $p:I\to J$ is called \emph{(right) cofinal} if for every $j$ in $J$ the over category ${p}_{j/}$ is nonempty and connected. If $p$ is cofinal then the morphism it determines, $\nu_{p,Y}:p^*Y\to Y$, is an isomorphism.

The word ind-object refers to objects of ind-categories. A \emph{simple} ind-object
is one indexed by the category with one object and one (identity) map. Note, that for any category $\cC$, there is a natural isomorphism between  $\cC$ and the full subcategory of $\Ind(\mcal{C})$ spanned by the simple objects. In the sequel we will abuse notation and consider objects and morphisms in $\cC$ as objects and morphisms in $\Ind(\cC)$, through this isomorphism.

If $T$ is a poset, then we view $T$ as a category which has a single morphism $u\to v$ iff $u\leq v$. Thus, a poset $T$ is filtered iff $T$ is non-empty, and for every $a,b$ in $T$ there exists an element $c$ in $T$ such that $c\geq a,b$. A filtered poeset will also be called \emph{directed}. A poset $T$ is called cofinite if for every element $x$ in $T$ the set $T_x:=\{z\in T| z \leq x\}$ is finite.

Let $\mcal{C}$ be a category with finite colimits and $M$ a class of morphisms in $\mcal{C}$. If $I$ is a small category and $F:X\to Y$ a morphism in $\mcal{C}^I$, then:
\begin{enumerate}
\item The map $F$ will be called a \emph{level-wise $M$-map}, if for every $i\in I$ the morphism $X_i\to Y_i$ is in $M$. We will denote this by $F\in \Lw(M)$.
\item The map $F$ will be called a \emph{cospecial} $M$-\emph{map}, if $I$ is a cofinite poset and for every $t\in I$ the natural map
$$X_t\coprod_{\colim_{s<t} X_s} \colim_{s<t} Y_s \to Y_t  $$
is in $M$. We will denote this by $F\in \coSp(M)$.
\end{enumerate}

We denote by $\R(M)$ the class of morphisms in $\mcal{C}$ that are retracts of morphisms in $M$. We denote by $M^{\perp}$ (resp. ${}^{\perp}M$) the class of morphisms in $\mcal{C}$ having the right (resp. left) lifting property with respect to all the morphisms in $M$.

We denote by $\Lw^{\cong}(M)$ the class of morphisms in $\Ind(\mcal{C})$ that are \textbf{isomorphic} to a morphism that comes from a natural transformation which is a levelwise $M$-map. The maps in $\Lw^{\cong}(\cW)$ are called \emph{essentially levelwise weak equivalences} by Isaksen \cite{Isa}.
We denote by $\coSp^{\cong}(M)$ the class of morphisms in $\Ind(\mcal{C})$ that are \textbf{isomorphic} to a morphism that comes from a natural transformation which is a cospecial $M$-map.

\subsection{From a weak cofibration category to a quasi model category}

In \cite{BaSc2} the notion of an \emph{almost model category} was introduced. It is a weaker notion then a model category that was used as an auxiliary notion which was useful in showing that certain weak cofibration categories are ind-admissible. In this paper we also consider a similar auxiliary notion, and for the same purpose. However, it would be more convenient for us to consider a slightly stronger notion, which we call \emph{a quasi model category}. A quasi model category has the same structure as a model category and satisfies all the axioms of a model category, except (maybe) one third of the two out of three property for its weak equivalences. Namely:

\begin{define}\label{d:quasi model}
A quasi model category is a quadruple $(\cM,\cW,\cF,\cC)$ satisfying the following axioms:
\begin{enumerate}
\item $\cM$ is complete and cocomplete.
\item $\cW,\cF,\cC$ are subcategories of $\cM$ that are closed under retracts.
\item For every pair $X\xrightarrow{f} Z\xrightarrow{g} Y $ of composable morphisms in $\cC$, we have that $f,g\circ f\in\cW$ implies $g\in \cW$.
\item $\cC\cap \cW\subseteq{}^{\perp}\cF$  and $\cC\subseteq{}^{\perp}(\cF\cap\cW)$.
\item There exist functorial factorizations in $\cM$ into a map in $\cC\cap \cW$ followed by a map in $\cF$ and into a map in $\cC$ followed by a map in $\cF\cap \cW$.
\end{enumerate}
\end{define}

A quasi model category is clearly a weaker notion then a model category. We also recall the following notion from \cite{BaSc2}:
\begin{define}\label{d:weak_fib}
A \emph{weak cofibration category} is a category ${\cC}$ with an additional
structure of two subcategories
$${\cC of}, {\cW} \subseteq {\cC}$$
that contain all the isomorphisms such that the following conditions are satisfied:
\begin{enumerate}
\item ${\cC}$ has all finite limits.
\item ${\cW}$ has the two out of three property.
\item The subcategories ${\cC of}$ and ${\cC of}\cap {\cW}$ are closed under cobase change.
\item Every map $A\to B $ in ${\cC}$ can be factored as $A\xrightarrow{f} C\xrightarrow{g} B $,
where $f$ is in ${\cC of}$ and $g$ is in ${\cW}$.
\end{enumerate}
The maps in ${\cC of}$ are called \emph{cofibrations}, and the maps in ${\cW}$ are called \emph{weak equivalences}.
\end{define}

\begin{define}\label{d:almost_admiss_dual}
A weak cofibration category $(\cC,\cW,\cC of)$ is called
\begin{enumerate}
  \item ind-admissible, if the class $\Lw^{\cong}(\cW)$, of morphisms in $\Ind(\cC)$, satisfies the two out of three property.
  \item almost ind-admissible, if the class $\Lw^{\cong}(\cW)$, of morphisms in $\Ind(\cC)$, satisfies the following portion of the two out of three property:

For every pair $X\xrightarrow{f} Z\xrightarrow{g} Y $ of composable morphisms in $\Ind(\cC)$ we have:
\begin{enumerate}
\item If $f,g$ belong to $\Lw^{\cong}(\cW)$ then $g\circ f\in \Lw^{\cong}(\cW)$.
\item If $f,g\circ f$ belong to $\Lw^{\cong}(\cW)$ then $g\in \Lw^{\cong}(\cW)$.
\end{enumerate}
\end{enumerate}
\end{define}

\begin{thm}[{\cite[Theorem 3.14]{BaSc2}}]\label{t:almost_model_dual}
Let $(\cC,\mcal{W},\cC of)$ be a small almost ind-admissible weak cofibration category.
Then there exists a quasi model category structure on $\Ind(\cC)$ such that:
\begin{enumerate}
\item The weak equivalences are $\mathbf{W} := \Lw^{\cong}(\mcal{W})$.
\item The fibrations are $\mathbf{F} := (\cC of\cap \mcal{W})^{\perp} $.
\item The cofibrations are $\mathbf{C} := \R(\coSp^{\cong}(\cC of))$.
\end{enumerate}
Furthermore, we have $\mathbf{F}  \cap \mathbf{W}=  \cC^{\perp}$ and $\mathbf{C}\cap\mathbf{W} = \R(\coSp^{\cong}(\cC of\cap{\cW})).$
\end{thm}

\begin{rem}\label{t:model}
If, in Theorem \ref{t:almost_model_dual}, the weak cofibration category $(\cC,\cW,\cC of)$ is also ind-admissible, then the quasi model structure on $\Ind(\cC)$ described there is clearly a model structure.
\end{rem}

The following notion is due to to Dwyer-Kan \cite{DwKa} and was further developed by Barwick-Kan \cite{BaKa}:
\begin{define}\label{d:rel}
A relative category is a pair $(\cC,\cW)$, consisting of a category
$\cC$, and a subcategory $\cW\subseteq \cC$ that contains all the
isomorphisms and satisfies the two out of three property. $\cW$ is called the subcategory of weak equivalences.
\end{define}
We now recall the notions of left and right proper maps in relative categories and the relation of these concepts to the almost ind-admissibility condition, appearing in Theorem \ref{t:almost_model_dual}.

\begin{define}\label{d:proper}
Let $(\cC,\cW)$ be a relative category. A map $f:A\to B$ in $\cC$ will be called:
\begin{enumerate}
\item Left proper, if for every pushout square of the form
\[
\xymatrix{A\ar[d]^i\ar[r]^f & B\ar[d]^j\\
C\ar[r] & D}
\]
such that $i$ is a weak equivalence, the map $j$ is also a weak equivalence.
\item Right proper, if for every pullback square of the form
\[
\xymatrix{C\ar[d]^j\ar[r] & D\ar[d]^i\\
A\ar[r]^f & B}
\]
such that $i$ is a weak equivalence, the map $j$ is also a weak equivalence.
\end{enumerate}
We denote by $LP$ the class of left proper maps in $\cC$ and by $RP$ the class of right proper maps in $\cC$.
\end{define}

If $\cC$ is a category and $M,N$ are classes of morphisms in $\cC$, we will denote by $\Mor({\cC}) = {M}\circ {N}$ the assertion that every map $A\to B $ in ${\cC}$ can be factored as $A\xrightarrow{f} C\xrightarrow{g} B $,
where $f$ is in ${N}$ and $g$ is in ${M}$.
The following proposition is the main motivation for introducing the concepts of left and right proper morphisms:
\begin{prop}[{\cite[Proposition 3.7]{BaSc2}}]\label{p:compose}
Let $(\cC,\cW)$ be a relative category, and let $X\xrightarrow{f} Y\xrightarrow{g} Z $ be a pair of composable morphisms in $\Ind(\cC)$. Then:
\begin{enumerate}
\item If $\cC$ has finite limits and colimits, and $\Mor(\cC)=RP\circ LP$, then $f,g\in \Lw^{\cong}(\cW)$ implies that $g\circ f\in \Lw^{\cong}(\cW)$.
\item If $\cC$ has finite limits, and $\Mor(\cC)=RP\circ \cW$, then $g,g\circ f\in \Lw^{\cong}(\cW)$ implies that $f\in \Lw^{\cong}(\cW)$.
\item If $\cC$ has finite colimits, and $\Mor(\cC)=\cW\circ LP$, then $f,g\circ f\in \Lw^{\cong}(\cW)$ implies that $g\in \Lw^{\cong}(\cW)$.
\end{enumerate}
\end{prop}

\subsection{Left proper simplicial and monoidal weak cofibration categories}

In this subsection we recall the notions of a left proper, simplicial and monoidal quasi model categories and weak cofibration categories. If a small almost ind-admissible weak cofibration category possesses one of these notions, the induced quasi model structure on its ind-category, given by Theorem \ref{t:almost_model_dual}, possesses the corresponding notion. For more details the reader is referred to \cite{BaSc3}.

\begin{define}\label{d:r_proper}
Let $\cC$ be a quasi model category or a weak cofibration category. Then $\cC$ is called left proper if for every pushout square of the form
\[
\xymatrix{A\ar[d]^i\ar[r]^f & B\ar[d]^j\\
C\ar[r] & D}
\]
such that $f$ is a cofibration and $i$ is a weak equivalence, the map $j$ is a weak equivalence.
\end{define}

The following proposition is shown in \cite[Corollary 3.3]{BaSc3}, based on the proof of \cite[Theorem 4.15]{Isa}:

\begin{prop}\label{c:r_proper}
Let $\cC$ be a small left proper almost ind-admissible weak cofibration category.
Then with the quasi model structure defined in Theorem \ref{t:almost_model_dual}, $\Ind(\cC)$ is a
left proper quasi model category.
\end{prop}

We can define the notion of a Quillen adjunction between quasi model categories, in two different ways, just as in the case of model categories:

\begin{define}[{\cite[Corollary 5.4]{BaSc3}}]\label{d:Qfunc}
Let $\cC$,$\cD$ be quasi model categories, and let
$$F:\cC\adj\cD:G$$
be an adjunction. We say that this adjunction is a \emph{Quillen pair} if one of the following equivalent conditions is satisfied:
\begin{enumerate}
\item The functor $F$ preserves cofibrations and trivial cofibrations.
\item The functor $G$ preserves fibrations and trivial fibrations.
\end{enumerate}
In this case we say that $F$ is a \emph{left Quillen functor} and $G$ is a \emph{right Quillen functor}.
\end{define}

\begin{define}\label{d:prolong}
Let $\cB$,$\cC$,$\cD$ be categories and let $(-)\otimes (-):\cB\times \cC\to \cD$ be a bifunctor.
We have a naturally induced prolongation of $\otimes$ to a bifunctor (which we also denote by $\otimes$)
$$(-)\otimes (-):\Ind(\cB)\times \Ind(\cC)\to \Ind(\cD).$$
If $B=\{B_i\}_{i\in I}$ is an object in $\Ind(\cB)$ and $C=\{C_j\}_{j\in J}$ is an object in $\Ind(\cC)$, then  $B\otimes C$ is the object in $\Ind(\cD)$ given by the diagram
$$\{B_{i}\otimes C_{j}\}_{(i,j)\in I\times J}.$$
\end{define}

\begin{define}\label{d:monoidal}
Let $(\cM,\otimes,I)$ be a symmetric monoidal category which is also a quasi model category (resp. weak cofibration category). We say that $\cM$, with this structure, is a monoidal quasi model category (resp. monoidal weak cofibration category) if the following conditions are satisfied:
\begin{enumerate}
\item The functor $\otimes:\cM\times \cM\to \cM$ is a part of a two variable adjunction (resp. commutes with finite colimits in every variable separately).
\item For every cofibration $j:X\to Y$ in $\cM$ and every cofibration $i:L\to K$ in $\cM$ the induced map
$$X \otimes K \coprod_{X \otimes L}Y\otimes L\to Y\otimes K$$
is a cofibration (in $\cM$), which is acyclic if either $i$ or $j$ is.
\item $I$ is a cofibrant object in $\cM$.
\end{enumerate}
\end{define}

\begin{prop}[{\cite[Theorem 5.15]{BaSc3}}]\label{p:monoidal}
Let $(\cM,\ten,I)$ be a small almost ind-admissible monoidal weak cofibration category. Then with the quasi model structure described in Theorem \ref{t:almost_model_dual} and with the natural prolongation of $\ten$ (see Definition \ref{d:prolong}), $\Ind(\cM)$ is a monoidal quasi model category.
\end{prop}

\begin{define}\label{d:simplicial}
Let $\cC$ be a quasi model category category which is also tensored over the category of simplicial sets $\cS$. We say that $\cC$, with this structure, is a simplicial quasi model category, if the following conditions are satisfied:
\begin{enumerate}
  \item The action $(-)\otimes (-):\cS\times \cC\to \cC$ is closed, that is, it is part of a two variable adjunction.
  \item For every cofibration $j:X\to Y$ in $\cS$ and every cofibration $i:L\to K$ in $\cC$ the induced map
$$X \otimes K \coprod_{X \otimes L}Y\otimes L\to Y\otimes K$$
is a cofibration (in $\cC$), which is acyclic if either $i$ or $j$ is.
\end{enumerate}
\end{define}

Let $\cS_f$ denote the category of \textbf{finite} simplicial sets, that is, $\cS_f$ is the full subcategory of $\cS$ spanned by simplicial sets having only a finite number of non-degenerate simplicies. There is a natural equivalence of categories $\Ind(\cS_f)\xrightarrow{\sim}\cS$, given by taking colimits (see \cite{AR}). We say that a map in $\cS_f$ is a cofibration or a weak equivalence, if it is so as a map in $\cS$.

\begin{define}\label{d:tensored}
Let $\cC$ be a weak cofibration category which is also tensored over $\cS_f$. We say that $\cC$, with this structure, is a simplicial weak cofibration category if the following conditions are satisfied:
\begin{enumerate}
  \item The action $(-)\otimes (-):\cS_f\times \cC\to \cC$ is \emph{weakly closed}, that is, it commutes with finite colimits in every variable separately.
\item For every cofibration $j:X\to Y$ in $\cS_f$ and every cofibration $i:L\to K$ in $\cC$ the induced map
$$X \otimes K \coprod_{X \otimes L}Y\otimes L\to Y\otimes K$$
is a cofibration (in $\cC$), which is acyclic if either $i$ or $j$ is.
\end{enumerate}
\end{define}

\begin{prop}[{\cite[Theorem 5.23]{BaSc3}}]\label{p:simplicial}
Let $\cC$ be a small almost ind-admissible simplicial weak cofibration category. Then with the quasi model structure described in Theorem \ref{t:almost_model_dual}  and with the natural prolongation of the action (see Definition \ref{d:prolong}), $\Ind(\cC)$ is a simplicial quasi model category.
\end{prop}

\section{Quasi model categories}\label{s:left quasi model}

In this section we develop some theory for quasi model categories. We begin by generalizing some theorems from model category theory to quasi model categories. These results will be used in the next section to show the ind-admissibility of certain weak cofibration categories. All the proofs given in the beginning of this section are straightforward generalizations of proofs appearing in \cite{Hov}.

We first prove the following generalization of Ken Brown’s Lemma:
\begin{lem}\
Let $\cD$ be a quasi model category and let $\cC$ be a relative category. Suppose $G:\cD\to\cC$ is a functor which takes trivial fibrations between fibrant objects to weak equivalences. Then $G$ takes all weak equivalences between fibrant objects to weak equivalences.
\end{lem}

\begin{proof}
Let $f:A\to B$ be a weak equivalence between fibrant objects in $\cC$. We factor the map $(id_A,f):A\to A\times B$ in the quasi model category $\cC$ into a trivial cofibration followed by a fibration $A\xrightarrow{q} C\xrightarrow{p} A\times B$. From the following pullback diagram:
$$\xymatrix{A\times B\ar[r]\ar[d] & A\ar[d]\\
             B\ar[r] & \ast}$$
it follows that the natural maps $\pi_0:A\times B\to A$ and $\pi_1:A\times B\to B$ are fibrations. We have
$$(\pi_1\circ p)\circ q=\pi_1\circ(p\circ q)=\pi_1\circ (id_A,f)=f.$$
Since $f,q$ are weak equivalences and $\cC$ is a quasi model category, it follows that $\pi_1\circ p$ is also a weak equivalence in $\cC$. Similarly one shows that $\pi_0\circ p$ is a weak equivalence in $\cC$. Thus $\pi_0\circ p$ and $\pi_1\circ p$ are trivial fibrations between fibrant objects. It follows that $G(\pi_0\circ p)$ and $G(\pi_1\circ p)$ are weak equivalences in $\cD$. Since
$$G(\pi_0\circ p)\circ G(q)=G(\pi_0\circ p\circ q)=G(\pi_0\circ (id_A,f))=G(id_A)=id_{G(A)}$$
is also a weak equivalence in $\cD$, and $\cD$ is a relative category, it follows that $G(q)$ is a weak equivalence in $\cD$. Since $\cD$ is a quasi model category the weak equivalences in $\cD$ are closed under composition, so we have that
$$G(f)=G((\pi_1\circ p)\circ q)=G(\pi_1\circ p)\circ G(q)$$
is a weak equivalence in $\cD$.
\end{proof}

\begin{cor}\label{c:Brown}
Consider a Quillen pair:
$$F:\cC\adj \cD:G,$$
where $\cC$ is a model category and $\cD$ is a quasi model category. Then $G$ takes weak equivalences between fibrant objects to weak equivalences.
\end{cor}

\begin{proof}
This follows from the previous lemma and Definition \ref{d:Qfunc}.
\end{proof}

We now come to our main result of this section.
\begin{prop}\label{p:simplicial_right}
Let $\cC$ be a simplicial quasi model category (see Definition \ref{d:simplicial}). Then a map in $\cC$, between fibrant cofibrant objects, is a weak equivalence iff it is a simplicial homotopy equivalence.
\end{prop}

\begin{proof}
Let $f:A\to B$ be a weak equivalence between fibrant cofibrant objects. The left action of $\cS$ on $\cC$ is a bifunctor $\otimes:\cS\times\cC\to\cC$ which is part of a two variable adjunction, denoted $(\otimes,\Map,\hom)$. It follows that for every cofibrant object $X$ of $\cC$, we have a Quillen adjunction (see Definition \ref{d:Qfunc})
$$(-)\otimes X:\cS\to \cC:\Map(X,-).$$
Since $A$ and $B$ are fibrant, it follows from Corollary \ref{c:Brown} that $\Map(X,f):\Map(X,A)\to\Map(X,B)$ is a weak equivalence in $\cS$. In particular it induces isomorphism on connected components
$f_*:\pi_0\Map(X,A)\to\pi_0\Map(X,B)$.

Taking $X=B$, we find a map $g:B\to A$ in $\cC$ such that $f\circ g\sim id_B$. It follows that $f\circ g\circ f\sim f$, so taking $X=A$, we find that $g\circ f\sim id_A$. Thus $f$ is a simplicial homotopy equivalence.

Let $f:A\to B$ be a simplicial homotopy equivalence between fibrant cofibrant objects. We factor $f$, in the quasi model category $\cC$, into a trivial cofibration followed by a fibration $A\xrightarrow{g}C\xrightarrow{p}B$. Since the weak equivalences in $\cD$ are closed under composition, it is enough to show that $p$ is a weak equivalence.

Let $f':B\to A$ be a simplicial homotopy inverse to $f$. Since $f\circ f'\sim id_B$ and $\Map(B,B)$ is a Kan-complex, we get that there exists a homotopy $H:\Delta^1\otimes B\to B$ such that $H|_{\Delta^{\{1\}}\otimes B}=id_B$ and $H|_{\Delta^{\{0\}}\otimes B}=f\circ f'$.

The map $\Delta^{\{0\}}\to \Delta^1$ is an acyclic cofibration in $\cS$, and the map $\phi\to B$ is a cofibration in $\cC$. Since  $\cC$ is a simplicial quasi model category, it follows that the map $\Delta^{\{0\}}\otimes B\to \Delta^1\otimes B$ is an acyclic cofibration in $\cC$.
Thus the following diagram:
$$\xymatrix{\Delta^{\{0\}}\otimes B\ar[r]^{g\circ f'}\ar[d]^{i_1} & C\ar[d]^p\\
           \Delta^1\otimes B \ar[r]^H & B,}$$
has a lift $H':\Delta^1\otimes B\to C$. We define $q:B\to C$ to be the composition
$$\xymatrix{B\cong \Delta^{\{1\}}\otimes B\ar[r] & \Delta^1\otimes B \ar[r]^{H'} & C.}$$
Then $p\circ q=id_B$, and $H'$ is a simplicial homotopy between $g\circ f'$ and $q$.

The object $C$ is also fibrant cofibrant and $g$ is a weak equivalence. As we have shown in the beginning of the proof, it follows that $g$ is a homotopy equivalence. Let $g':C\to A$ be a simplicial homotopy inverse to $g$. Then we have
$$p\sim p\circ g\circ g'\sim f\circ g',$$
so it follows that
$$q\circ p\sim (g\circ f')\circ(f\circ g')\sim id_C.$$
Since $\Map(C,C)$ is a Kan-complex, we get that there exists a homotopy $K:\Delta^1\otimes C\to C$ such that $K|_{\Delta^{\{0\}}\otimes C}=id_C$ and $K|_{\Delta^{\{1\}}\otimes C}=q\circ p$.

The maps $\Delta^{\{0\}}\to \Delta^1$ and $\Delta^{\{1\}}\to \Delta^1$ are acyclic cofibrations in $\cS$, and the map $\phi\to C$ is a cofibration in $\cC$. Since  $\otimes:\cS\times\cC\to\cC$ is a left Quillen bifunctor, it follows that the maps $\Delta^{\{0\}}\otimes C\to \Delta^1\otimes C$ and $\Delta^{\{1\}}\otimes C\to \Delta^1\otimes C$ are acyclic cofibrations in $\cC$. The map
$$\Delta^{\{0\}}\otimes C\to \Delta^1\otimes C\xrightarrow{K} C$$
is just $id_C$ so it is also a weak equivalence in $\cC$. Since $\cC$ is a quasi model category, it follows that $K$ is a weak equivalence in $\cC$. Since $\cC$ is a quasi model category the weak equivalences in $\cC$ are closed under composition. Since $q\circ p$ is just the composition
$$\Delta^{\{1\}}\otimes C\to \Delta^1\otimes C\xrightarrow{K} C,$$
we get that $q\circ p$ is a weak equivalence in $\cC$. The following diagram:
$$\xymatrix{C\ar[r]^=\ar[d]^p & C\ar[r]^=\ar[d]^{q\circ p} & C\ar[d]^p \\
             B\ar[r]^q &C \ar[r]^p & B}$$
shows that $p$ is a retract of $q\circ p$, which finishes our proof.
\end{proof}

\begin{define}
Let $\cC$ be a quasi model category.
\begin{enumerate}
\item A fibrant replacement functor in $\cC$ is an endofunctor $R:\cC\to\cC$ with image in the fibrant objects of $\cC$, together with a natural transformation $id_{\cC}\to R$ which is a level-wise weak equivalence.
\item A cofibrant replacement functor in $\cC$ is an endofunctor $Q:\cC\to\cC$ with image in the cofibrant objects of $\cC$, together with a natural transformation $Q\to id_{\cC}$ which is a level-wise weak equivalence.
\end{enumerate}
\end{define}

\begin{lem}\label{l:right_preserve}
Let $\cC$ be a quasi model category.
Then the following hold:
\begin{enumerate}
\item Any fibrant replacement functor preserves weak equivalences. That is, if $R:\cC\to \cC$ is a fibrant replacement functor and $X\to Y$ is a weak equivalence in $\cC$, then $R(X)\to R(Y)$ is a weak equivalence.
\item Any cofibrant replacement functor reflects weak equivalences. That is, if $Q:\cC\to \cC$ is a cofibrant replacement functor and $X\to Y$ is a map in $\cC$ such that  $Q(X)\to Q(Y)$ is a weak equivalence, then $X\to Y$ is a weak equivalence.
\end{enumerate}
\end{lem}

\begin{proof}
\begin{enumerate}
\item
Let $X\to Y$ be a weak equivalence in $\cC$. By considering the following diagram
$$\xymatrix{X \ar[dr]\ar[r]^{\sim}\ar[d]_{\sim} & Y\ar[d]^{\sim}\\
            R(X) \ar[r] & R(Y)}$$
and the definition of a quasi model category, it follows that $R(X)\to R(Y)$ is a weak equivalence.
\item
Let $X\to Y$ be a  map in $\cC$ such that  $Q(X)\to Q(Y)$ is a weak equivalence. By considering the following diagram
$$\xymatrix{Q(X) \ar[dr]\ar[r]^{\sim}\ar[d]_{\sim} & Q(Y)\ar[d]^{\sim}\\
            X \ar[r] & Y}$$
and the definition of a quasi model category, it follows that $X\to Y$ is a weak equivalence.

\end{enumerate}
\end{proof}
Proposition \ref{p:simplicial_right} has the following interesting corollary:
\begin{prop}\label{p:iff_model}
Let $\cC$ be a simplicial quasi model category. Then $\cC$ is a model category iff there exists a fibrant replacement functor that reflects weak equivalences and a cofibrant replacement functor that preserves weak equivalences such that either one of the following holds:
\begin{enumerate}
\item This fibrant replacement functor preserves cofibrant objects.
\item This cofibrant replacement functor preserves fibrant objects.
\end{enumerate}
\end{prop}

\begin{proof}
Clearly, if $\cC$ is an actual model category, then there exist such cofibrant and fibrant replacement functors in $\cC$. Simply choose the functors obtained by any functorial factorization in the model category $\cC$.

Conversely, let $R:\cC\to \cC$ be a fibrant replacement functor that reflects weak equivalences and let $Q:\cC\to \cC$ be a cofibrant replacement functor that preserves weak equivalences. With out loss of generality, we assume that $R$ preserves cofibrant objects. We wish to show that the weak equivalences in $\cC$ are precisely the maps $X\to Y$ in $\cC$ such that $R(Q(X))\to R(Q(Y))$ is a simplicial homotopy equivalence in $\cC$. This will show that the weak equivalences in $\cC$ satisfy the two out of three property, and thus $\cC$ is a model category.

So let $X\to Y$ be a weak equivalence in $\cC$. By our assumption on $Q$, we have that $Q(X)\to Q(Y)$ is a weak equivalence in $\cC$. By Lemma \ref{l:right_preserve}, we have that $R(Q(X))\to R(Q(Y))$ is a weak equivalence in $\cC$. Since $R$ preserves cofibrant objects, we see that $R(Q(X))$ and $R(Q(Y))$ are fibrant cofibrant. It thus follows from Proposition \ref{p:simplicial_right} that $R(Q(X))\to R(Q(Y))$ is a simplicial homotopy equivalence in $\cC$.

Conversely, let $X\to Y$ be a map in $\cC$ such that $R(Q(X))\to R(Q(Y))$ is a simplicial homotopy equivalence in $\cC$. Since $R(Q(X))$ and $R(Q(Y))$ are fibrant cofibrant objects, we have by Proposition \ref{p:simplicial_right} that $R(Q(X))\to R(Q(Y))$ is a weak equivalence in $\cC$. By our assumption on $R$, we have that $Q(X)\to Q(Y)$ is a weak equivalence in $\cC$. By Lemma \ref{l:right_preserve}, we have that $X\to Y$ is a weak equivalence in $\cC$.
\end{proof}

\section{Criteria for the two out of three property}\label{s:Criteria}

In this section we prove our main results of this paper, namely, we prove theorems giving sufficient conditions on a weak cofibration category that insure its ind-admissibility. In the first subsection we prove two preliminary criteria for ind-admissibility (Theorems \ref{t:admiss_iff} and \ref{t:main}), and in the second we prove our main result (Theorem \ref{t:main2}).

For the sake of clarity, let us explain briefly the role played by quasi model categories and the theory developed in the previous section in the proof of our main result. We begin with a weak cofibration category $\cM$ that satisfies certain hypothesis and we wish to show that it is ind-admissible. We cannot use Proposition \ref{p:compose} for this because we don't know how to obtain factorizations of the form $\Mor(\cM)=RP\circ \cW$. However, we can easily obtain factorizations of the form $\Mor(\cM)=RP\circ LP$ and $\Mor(\cM)=\cW\circ LP$. These can be obtained by a certain mapping cylinder factorization, as will be explained below. Thus, using Proposition \ref{p:compose}, we can deduce that $\cM$ is almost ind-admissible, and there is an induced \textbf{quasi} model structure on $\Ind(\cM)$. Knowing this, we can use the previously developed theory for quasi model categories on $\Ind(\cM)$. Doing so, together with other lines of proof that depend on the specific properties of the given weak cofibration category, we are able to show that $\cM$ is not only almost ind-admissible but also ind-admissible.

\subsection{Preliminary criteria}
\begin{thm}\label{t:admiss_iff}
Let $(\cM,\cW,\cC)$ be a small almost ind-admissible simplicial weak cofibration category. Suppose that $\mcal{M}$ has functorial factorizations into a cofibration followed by a weak equivalence. Then $\cM$ is ind-admissible iff there exists a fibrant replacement functor in the quasi model category $\Ind(\cM)$ given by Theorem \ref{t:almost_model_dual}, that reflects weak equivalences and preserves cofibrant objects.
\end{thm}

\begin{proof}
By Theorem \ref{t:almost_model_dual}, there exists an induced quasi model category structure on $\Ind(\cM)$. By Proposition \ref{p:simplicial} this quasi model category is also simplicial. Using Proposition \ref{p:iff_model}, we see that in order to prove our theorem it is enough to show that there exists a cofibrant replacement functor on $\Ind(\cM)$ that preserves weak equivalences.

By \cite{BaSc} Theorem 5.7, we can use the functorial factorization in $\cM$ into a cofibration followed by a weak equivalence to produce a functorial factorization in $\Ind(\cM)$ into a morphism in $\coSp^{\cong}(\cC)$ followed by a morphism in $\Lw^{\cong}(\cW)$.
For any object $X$ in $\Ind(\cM)$, we can factor the unique map $\phi\to X$ using this functorial factorization and obtain a cofibrant replacement functor $Q:\Ind(\cM)\to \Ind(\cM).$
It remains to show that $Q$ preserves the weak equivalences in $\Ind(\cM)$, that is, that $Q$ preserves maps in $\Lw^{\cong}(\cW)$. So let $f:X\to Y$ be a map in $\Lw^{\cong}(\cW)$. Then $f$ is isomorphic to a natural transformation $f':X'\to Y'$ that is a levelwise $\cW$ map, on a (small) cofiltered category $T$. By \cite[Proposition 8.1.6]{SGA4-I}, we can choose a cofinite directed set $J$ and a cofinal functor $p:J\to T$. Then $p^*f':p^*X'\to p^*Y'$ is a natural transformation that is a levelwise $\cW$ map, on the cofinite cofiltered set $J$. By the construction of the functorial factorization given in \cite{BaSc} Theorem 5.7, we see that applying the functor $Q$ to $f:X\to Y$ is isomorphic to the Reedy construction on $p^*X'\to p^*Y'$ (see \cite{BaSc} Definition 4.3). We have a diagram in $\cM^J$
$$\xymatrix{(p^*X')_{Reedy}\ar[r]^{\Lw(\mcal{W})}\ar[d] & p^*X' \ar[d]^{\Lw(\mcal{W})}\\
 (p^*Y')_{Reedy}\ar[r]^{\Lw(\mcal{W})} & p^*Y'.}$$
It follows that the map between the Reedy constructions $(p^*X)_{Reedy}\to (p^*Y)_{Reedy}$ is in $\Lw(\cW)$, and thus, that the map $Q(X)\to Q(Y)$ is in $\Lw^{\cong}(\cW)$.
\end{proof}

Let $(\cM,\cW,\cC)$ be a small almost ind-admissible simplicial weak cofibration category with functorial factorizations into a cofibration followed by a weak equivalence. By Theorem \ref{t:admiss_iff} in order to show that $\cM$ is ind-admissible it is enough to show that there exists a fibrant replacement functor in the quasi model category $\Ind(\cM)$ that reflects weak equivalences and preserves cofibrant objects. We wish to formulate sufficient conditions for such a fibrant replacement functor to exist. In this paper we only do this assuming a rather strong assumption, namely, that \emph{every} object in $\Ind(\cM)$ is already fibrant. In this case we can clearly choose $\id_{\Ind(\cM)}$ as our fibrant replacement functor.

\begin{lem}\label{l:everything_fib}
Let $\cM$ be a small almost ind-admissible weak cofibration category, and consider the quasi model structure induced on $\Ind(\cM)$ by Theorem \ref{t:almost_model_dual}. Then every object in $\Ind(\cM)$ is fibrant iff every acyclic cofibration in $\cM$ admits a left inverse (that is, a retracting map).
\end{lem}

\begin{proof}
Suppose that every acyclic cofibration in $\cM$ admits a left inverse. Let $X$ be an object in $\Ind(\cM)$. We need to show that $X$ is fibrant. Let $i:A\to B$ be an acyclic cofibration in $\cM$. It is enough to show that every diagram of the form
$$\xymatrix{A\ar[d]^i\ar[r] & X\\
             B }$$
has a lift $h:B\to X$. Let $p:B\to A$ be a left inverse to $i$, that is, we have $p\circ i=id_A$. Then we can choose $h$ to be the composition
$B\xrightarrow{p}A\to X.$

Now suppose that every object in $\Ind(\cM)$ is fibrant. Let $i:A\to B$ be an acyclic cofibration in $\cM$. We need to show that $i$ admits a left inverse. Consider the diagram
$$\xymatrix{A\ar[d]^i\ar[r]^= & A.\\
             B }$$
Since the object $A$ is fibrant in $\Ind(\cM)$, we have a lift $p:B\to A$. Then $p$ is a left inverse to $i$.
\end{proof}

The last lemma leads to the following definition:
\begin{define}
An object $D$ in a weak cofibration category $\cM$ is called fibrant if for every acyclic cofibration $A\to B$ in $\cM$ and every diagram of the form
$$
\xymatrix{ A \ar[d]\ar[r] & D\\
 B & }
$$
there is a lift $B\to D$.
\end{define}

\begin{lem}\label{l:everything_fib_small}
Let $\cM$ be a weak cofibration category. Then every object in $\cM$ is fibrant iff every acyclic cofibration in $\cM$ admits a left inverse.
\end{lem}

\begin{proof}
Exactly like the proof of Lemma \ref{l:everything_fib}.
\end{proof}

\begin{thm}\label{t:main}
Let $\mcal{M}$ be a small almost ind-admissible simplicial weak cofibration category. Suppose that $\mcal{M}$ has functorial factorizations into a cofibration followed by a weak equivalence and that very object in $\cM$ is fibrant. Then $\cM$ is ind-admissible.
\end{thm}

\begin{proof}
By Theorem \ref{t:admiss_iff} in order to show that $\cM$ is ind-admissible it is enough to show that there exists a fibrant replacement functor in the quasi model category $\Ind(\cM)$ that reflects weak equivalences and preserves cofibrant objects.
By Lemmas \ref{l:everything_fib} and \ref{l:everything_fib_small}, every object in $\Ind(\cM)$ is fibrant. Thus, we can simply choose the identity functor on $\Ind(\cM)$ as our fibrant replacement functor.
\end{proof}

\begin{rem}
  The almost ind-admissibility condition on a weak cofibration category $\cM$, appearing in Theorems \ref{t:admiss_iff} and \ref{t:main}, can be verified by checking that $\cM$ has finite limits and factorizations of the form $\Mor(\cM)=RP\circ LP$ and $\Mor(\cM)=\cW\circ LP$ (see Proposition \ref{p:compose}).
\end{rem}

\subsection{Our main criterion: special weak cofibration categories}
\begin{define}\label{d:cylinder}
Let $(\mcal{M},\mcal{W},\mcal{C})$ be a simplicial weak cofibration category.

Let $f:A\to B$ be a morphism in $\cM$. We define the mapping cylinder of $f$ to be the pushout
$$\xymatrix{A\ar[r]^{i_0}\ar[d]^f & \Delta^1\otimes A\ar[d] \\
             B \ar[r]^j & C(f).}$$

We define a morphism $p:C(f)=B\coprod_{A}\Delta^1\otimes A\to B$ to be the one induced by the commutative square
$$\xymatrix{A\ar[dd]^{f}\ar[r]^{i_0} & \Delta^1\otimes A\ar[dd]^{}\ar[dr] & \\ & & \Delta ^0\otimes A\cong A, \ar[dl]^f\\
                 B \ar[r]^{=} & B & }$$
and we define a morphism $i:A\to C(f)=B\coprod_{A}\Delta^1\otimes A$ to be the composition
$${A\xrightarrow{i_1}  \Delta^1\otimes A \xrightarrow{}  C(f).}$$

Clearly $f=pi$, and we call this the mapping cylinder factorization.
\end{define}

\begin{lem}\label{l:fact}
Let $(\mcal{M},\mcal{W},\mcal{C})$ be a simplicial weak cofibration category such that every object in $\cM$ is cofibrant. Then the mapping cylinder factorization is a functorial factorization in $\cM$ into a cofibration followed by a weak equivalence.
\end{lem}

\begin{proof}
Since every object in $\cM$ is cofibrant, $\cM$ is a Brown category of cofibrant objects. Let $B$ be any object of $\cM$. Since $B$ is cofibrant and $\Delta^{\{0\}}\to \Delta^1$ is an acyclic cofibration, we get that
$$B\cong \Delta^{\{0\}}\otimes B\xrightarrow{i_0} \Delta^1\otimes B$$
is an acyclic cofibration. A similar argument shows that $$B\sqcup B\cong (\Delta^{\{0\}}\sqcup \Delta^{\{1\}})\otimes B\xrightarrow{(i_0,i_1)} \Delta^1\otimes B$$
 is a cofibration. Since the composition
$$B\xrightarrow{i_0} \Delta^1\otimes B\xrightarrow{\pi}\Delta^{\{0\}}\otimes B\cong B$$
is the identity, we get that $\pi$ is a weak equivalence by two out of three. We thus obtain that $(\Delta^1\otimes B,\pi,i_0,i_1)$ is a cylinder object for $B$. Now the result follows from Brown's factorization lemma \cite{Bro}.
\end{proof}

\begin{lem}\label{l:homotopy}
Let $(\mcal{M},\mcal{W},\mcal{C})$ be a simplicial weak cofibration category such that every object in $\cM$ is cofibrant. Then the map $jp$ in the mapping cylinder factorization is simplicialy homotopic to the identity map on $C(f)$ via a \emph{cofiber preserving homotopy}. That is, there exists a morphism $H: \Delta^1\otimes C(f)\to C(f)$ satisfying $H i_1=\id_{C(f)}$, $ H i_0=jp$ and such that the following diagram commutes:
$$\xymatrix{\Delta^1\otimes C(f)\ar[r]^{}\ar[d]^{H} & \Delta^0\otimes C(f)\cong C(f)\ar[d]^{p} \\
                 C(f)  \ar[r]^{p} & B.}$$
\end{lem}

\begin{proof}
Since $\cC$ is a simplicial weak cofibration category, we have that the functor $\Delta^1\otimes (-):\cC\to\cC$ commutes with pushouts. Thus, we have
$$\Delta^1 \otimes C(f)\cong
\Delta^1\otimes (B\sqcup_A (\Delta^1 \otimes A))\cong
\Delta^1\otimes B\sqcup_{\Delta^1\otimes A}\Delta^1\otimes(\Delta^1\otimes A)\cong$$
$$\cong\Delta^1\otimes B\sqcup_{\Delta^1\otimes A}(\Delta^1\times \Delta^1)\otimes A.$$
We define the homotopy $H$ to be the map
$$\Delta^1\otimes B\sqcup_{\Delta^1\otimes A}(\Delta^1\times \Delta^1)\otimes A\cong \Delta^1\otimes C(f)\to C(f)\cong B\sqcup_A \Delta^1 \otimes A,$$
induced by:
\begin{enumerate}
\item The natural map $\Delta^1\otimes B\to \Delta^0\otimes B\cong B$.
\item The natural map $\Delta^1\otimes A\to \Delta^0\otimes A\cong A$.
\item The map $(\Delta^1\times\Delta^1)\otimes A\to \Delta^1\otimes A$ induced by the simplicial map $\Delta^1\times \Delta^1\to \Delta^1$ that sends $(0,0),(0,1),(1,0)$ to $0$ and $(1,1)$ to $1$.
\end{enumerate}

It is not hard to verify that these maps indeed define a map between the pushout diagrams, and thus a map $H$ as required. It is also clear that this $H$ has the required properties.
\end{proof}

\begin{prop}\label{l:right_proper}
Let $(\mcal{M},\mcal{W},\mcal{C})$ be a simplicial weak cofibration category such that every object in $\cM$ is cofibrant. Suppose that every homotopy equivalence in the simplicial category $\cC$ is a weak equivalence. Then the map $p$ in the mapping cylinder factorization is right proper.
\end{prop}

\begin{proof}
Consider a pullback square in $\cC$

$$\xymatrix{
C(f)\times_B C \ar[d]^{k}\ar[r]^\theta & C(f)\ar[d]^p\\
C \ar[r]^g & B,}$$
such that $g$ is a weak equivalence. We need to show that $\theta$ is a weak equivalence. Since $p$ is a weak equivalence, it follows from the two out of three property for the weak equivalences in $\cC$ that it is enough to show that $k$ is a weak equivalence. We will show this by constructing a homotopy inverse to $k$. We define the morphism $l:C\to C(f)\times_B C$ to be the one induced by the commutative square

$$\xymatrix{C\ar[r]^{j\circ g}\ar[d]_{=} & C(f)\ar[d]^{p} \\
                C   \ar[r]^{g} & B.}$$
Since $j$ is a right inverse to $p$ this is indeed well defined.
It remains to show that $l$ is a homotopy inverse to $k$.

It is readily verified that $kl=\id_C$.
By Lemma \ref{l:homotopy}, the map $jp$ is simplicialy homotopic to the identity map on $C(f)$ via a cofiber preserving homotopy. That is, there exists a morphism $H: \Delta^1\otimes C(f)\to C(f)$ satisfying $H i_1=\id_{C(f)}$, $ H i_0=jp$ and such that the following diagram commutes:
$$\xymatrix{\Delta^1\otimes C(f)\ar[r]^{}\ar[d]^{H} & \Delta^0\otimes C(f)\cong C(f)\ar[d]^{p} \\
                 C(f)  \ar[r]^{p} & B.}$$
The canonical map
$\theta: C(f)\times_B C\to C(f)$
induces a map
$$\Delta^1\otimes\theta:  \Delta^1\otimes(C(f)\times_B C)\to \Delta^1\otimes C(f).$$ Consider the homotopy
$$\Phi: \Delta^1\otimes(C(f)\times_B C)\to C(f)\times_B C,$$ induced by the commutative square
$$\xymatrix{\Delta^1\otimes(C(f)\times_B C)\ar[rr]^{H\circ(\Delta^1\otimes\theta)}\ar[d] & & C(f)\ar[d]^{p} \\
                C   \ar[rr]^{g} & & B,}$$
the left vertical map being the composite
$$\Delta^1\otimes(C(f)\times_B C)\to \Delta^0\otimes(C(f)\times_B C)\cong C(f)\times_B C\to C.$$
From the definition it follows now that $\Phi i_0 = \id_{C(f)\times_B C}$ and $\Phi i_1= lk$, demonstrating that $lk$ is homotopic to $\id_{C(f)\times_B C}$.
\end{proof}

\begin{define}\label{d:special WCC}
A small simplicial weak cofibration category $\mcal{M}$ is called \emph{special} if the following conditions are satisfied:
\begin{enumerate}
\item $\cM$ has finite limits.
\item Every object in $\cM$ is fibrant and cofibrant.
\item A map in $\cM$ that is a homotopy equivalence in the simplicial category $\cM$ is also a weak equivalence.
\end{enumerate}
\end{define}

\begin{thm}\label{t:main2}
Let $\mcal{M}$ be a special weak cofibration category (see Definition \ref{d:special WCC}).
Then $\cM$ is left proper and ind-admissible.
\end{thm}

\begin{proof}
Since every object in $\cM$ is cofibrant, $\cM$ is a Brown category of cofibrant objects. It follows from \cite[Lemma 8.5]{GJ} that $\cM$ is left proper. According to Lemma \ref{l:fact}, the mapping cylinder factorization gives a functorial factorization in $\cM$ into a cofibration followed by a weak equivalence. Using condition 3 of Definition \ref{d:special WCC}, Proposition \ref{l:right_proper}, and the fact that $\cM$ is left proper, we see that the mapping cylinder factorization gives a factorization in $\cM$ of the form $$\Mor(\cM)= (RP\cap \cW)\circ LP.$$
Using condition 1 and Proposition \ref{p:compose}, we see that $\cM$ is almost ind-admissible (see Definition \ref{d:almost_admiss_dual}). Thus our theorem now follows from Theorem \ref{t:main}.
\end{proof}

\begin{thm}\label{t:main3}
Let $(\mcal{M},\mcal{W},\mcal{C})$ be a special weak cofibration category (see Definition \ref{d:special WCC}).
Then there exists a simplicial model category structure on $\Ind(\cM)$ such that:
\begin{enumerate}
\item The weak equivalences are $\mathbf{W} := \Lw^{\cong}(\cW)$.
\item The cofibrations are $\mathbf{C} := \R(\coSp^{\cong}(\cC))$.
\item The fibrations are $\mathbf{F} :=(\cC\cap \cW)^{\perp}$.
\end{enumerate}

Moreover, this model category is finitely combinatorial, with set of generating cofibrations $\cC$ and set of generating acyclic cofibrations $\cC\cap \cW$.

The model category $\Ind(\cM)$ has the following further properties:
\begin{enumerate}
\item Every object in $\Ind(\cM)$ is fibrant.
\item $\Ind(\cM)$ is proper.
\end{enumerate}
\end{thm}

\begin{proof}
The fact that the prescribed weak equivalences fibrations and cofibrations give a finitely combinatorial model structure on $\Ind(\cM)$, with set of generating cofibrations $\cC$ and set of generating acyclic cofibrations $\cC\cap \cW$ follows from Theorems \ref{t:almost_model_dual} and \ref{t:main2}. The fact that it is simplicial follows from Proposition \ref{p:simplicial}.

As for the further properties; the first one follows from Lemmas \ref{l:everything_fib} and \ref{l:everything_fib_small}. Right properness follows from the fact that every model category in which every object is
fibrant is right proper (see for example \cite[Proposition A.2.4.2]{Lur}) and left properness follows from the fact that $\cM$ is left proper (by Theorem \ref{t:main2}) and Proposition  \ref{c:r_proper}.
\end{proof}

\section{Compact metrizable spaces}\label{s:CM}
In the previous section we defined the notion of a special weak cofibration category, and showed that every special weak cofibration category is ind admissible. In this section we discuss a way to construct special weak cofibration categories.
We end by an example from the category of compact metrizable spaces.

\subsection{Constructing special weak cofibration categories}
Let $\cC$ be a small category with finite colimits. Suppose that $\cC$ is tensored over $\cS_f$ and the action is weakly  closed  (see Definition \ref{d:tensored} 1).
Note that $\cC$ is automatically enriched in $\cS$, with the defining property that for every $A,B\in\cC$ and $L\in S_f$ we have
$$\Hom_{\cS}(L,\Map_{\cC}(A,B))\cong \Hom_{\cC}(L\otimes A,B).$$

\begin{prop}\label{p:special}
Let $\cD$ be a simplicial weak cofibration category (see Definition \ref{d:tensored}), and let $J:\cC\to \cD$ be a functor  such that:
\begin{enumerate}
\item The functor $J$ commutes with finite colimits and the simplicial action.
\item The functor $J$ lands in the cofibrant objects of $\cD$.
\end{enumerate}

Then if we define a map $f$ in $\cC$ to be a {\em weak equivalence} or a {\em cofibration} if $J(f)$ is so in $\cD$, $\cC$ becomes a simplicial weak cofibration category in which every object is cofibrant. If, furthermore, the functor $J:\cC\to\cD$ is fully faithful and lands in the fibrant objects of $\cD$, then every object in $\cC$ is fibrant.
\end{prop}

\begin{proof}
It is easy to see that the cofibrations and the weak equivalences in $\cC$ are subcategories, the weak equivalences in $\cC$ have the two out of three property and the cofibrations and acyclic cofibrations in $\cC$ are closed under cobase change.

Since $J$ commutes with finite colimits and the simplicial action and $\cD$ is a simplicial weak cofibration category, it is not hard to see that for every cofibration $j:K\to L$ in $\cS_f$ and every cofibration $i:X\to Y$ in $\cC$ the induced map
$$K \otimes Y \coprod_{K \otimes X}L\otimes X\to L\otimes Y$$
is a cofibration in $\cC$, which is acyclic if either $i$ or $j$ is.

We now define for every morphism $f:A\to B$ in $\cC$ the mapping cylinder factorization of $f$
$$A\xrightarrow{i}B\coprod_{A}\Delta^1\otimes A\xrightarrow{p}B$$
just as in Definition \ref{d:cylinder}. Since $J$ lands in the cofibrant objects in $\cD$ it follows that every object of $\cC$ is cofibrant. It can thus be shown, just like in Lemma \ref{l:fact}, that the mapping cylinder factorization is a functorial factorization in $\cC$ into a cofibration followed by a weak equivalence.

It is easy to see that if the functor $J:\cC\to\cD$ is furthermore fully faithful and lands in the fibrant objects of $\cD$, then every object in $\cC$ is fibrant.
\end{proof}

The last proposition brings us close to a special weak cofibration category (see Definition \ref{d:special WCC}). We now wish to consider a specific example of a category $\cD$ and a functor $J$, as in the last proposition, that will indeed produce a special weak cofibration category structure on $\cC$.

Consider the enriched dual Yoneda embedding
$$Y:A\mapsto \Map_{\cC}(A,-):\cC\to (\cS^\cC)^{\op}.$$
We endow the category $\cS^\cC$ with the \emph{projective} model structure. In particular, this makes $(\cS^\cC)^{\op}$ into a weak cofibration category. The category
$\cS^\cC$, in the projective structure, is a simplicial model category (see for example \cite[Remark A.3.3.4]{Lur}).
Note, that the notion of a simplicial model category is self dual. That is, if  $\cA$ is a simplicial model category
then  $\cA^{\op}$ is a simplicial model category, with the natural dual simplicial structure.
In particular, it follows that $(\cS^\cC)^{\op}$ is a simplicial model category with
$$\otimes_{(\cS^\cC)^{\op}}:=\hom_{\cS^\cC}^{\op}:\cS\times (\cS^\cC)^{\op}\to (\cS^\cC)^{\op}.$$

\begin{lem}\label{l:commute}
The functor $Y:\cC\to (\cS^\cC)^{\op}$  commutes with finite colimits and the simplicial action.
\end{lem}

\begin{proof}
The fact that $Y$ commutes with finite colimits is clear.
It is left to show that there are coherent natural isomorphisms
$$\Map_{\cC}(K\otimes A,-)\cong K\otimes_{(\cS^\cC)^{\op}} \Map_{\cC}(A,-)$$
for $K\in \cS_f$ and $A\in \cC$.
Thus, for every $K\in \cS_f$ and $A,B\in \cC$, we need to supply an isomorphism
$$\Map_{\cC}(K\otimes A,B)\cong \hom_{\cS}(K, \Map_{\cC}(A,B)),$$
but this is clear.
\end{proof}

The following lemma is clear
\begin{lem}\label{l:lands in cofibrant}
The functor $Y:\cC\to(\cS^\cC)^{\op}$ lands in the cofibrant objects of $(\cS^\cC)^{\op}$ iff for every $A,B\in\cC$ the simplicial set $\Map_\cC(A,B)$ is a Kan complex.
\end{lem}

\begin{lem}\label{l:lands in fibrant}
The functor $Y:\cC\to(\cS^\cC)^{\op}$ lands in the fibrant objects of $(\cS^\cC)^{\op}$.
\end{lem}

\begin{proof}
For every $A\in \cC$ and every $X\in\cS$, we denote
$$F_X^A:=\Map_{\cC}(A,-)\times X\in \cS^\cC.$$
By \cite[Remark A.3.3.5]{Lur}, for every cofibration $X\to Y$ in $\cS$ the map $F_X^A\to F_Y^A$ is a cofibration in the projective structure on $\cS^\cC.$ In particular, the map $F_{\phi}^A\to F_{\Delta^0}^A$ is a cofibration in $\cS^\cC$, and thus
$$F_{\Delta^0}^A=\Map_{\cC}(A,-)\in \cS^\cC$$
is cofibrant.
\end{proof}

We now come to our main conclusion.
\begin{thm}\label{t:special}
Let $\cC$ be a small category with finite limits and colimits, that is tensored over $\cS_f$ and the action is weakly closed  (see Definition \ref{d:tensored} 1). Suppose that for every $A,B\in\cC$ the simplicial set $\Map_\cC(A,B)$ is a Kan complex. Consider the enriched dual Yoneda embedding
$$Y:A\mapsto \Map_{\cC}(A,-):\cC\to (\cS^\cC)^{\op}.$$
We endow the category $\cS^\cC$ with the projective model structure, and we define a map $f$ in $\cC$ to be a weak equivalence or a cofibration if $Y(f)$ is a weak equivalence or a cofibration in $(\cS^\cC)^{\op}$. Then $\cC$ is a special weak cofibration category.
\end{thm}

\begin{proof}
It is well known that $Y$ is fully faithful. Thus, using Proposition \ref{p:special} and Lemmas \ref{l:commute}, \ref{l:lands in cofibrant}, \ref{l:lands in fibrant}, we are only left to show that a map in $\cC$ that is a homotopy equivalence in the simplicial category $\cC$ is also a weak equivalence.
But this follows from the definition of weak equivalences in $\cC$ and \cite[Proposition 1.2.4.1]{Lur}.
\end{proof}

\begin{rem}\label{r:equiv cof}
Let $\cC$ be as in Theorem \ref{t:special} and let $i: A\to B$ be a map in $\cC$. Then $i$ is a cofibration iff for every $D\in\cC$ and every acyclic cofibration $K\to L$ in $\cS_f$ and every commutative diagram of the form
$$
\xymatrix{
K \ar[d] \ar[r] & Map_{\cC}(B,D)\ar[d]^{i^*}\\
L\ar[r] & Map_{\cC}(A,D),
}$$
there exists a lift $L\to Map_{\cC}(B,D)$.
By adjointness this is equivalent to requiring that for every $D\in\cC$ and every acyclic cofibration $K\to L$ in $\cS_f$ and every diagram of the form

$$
\xymatrix{ K\otimes B\coprod_{K\otimes A}L\otimes A\ar[d]\ar[r] & D\\
 L\otimes B & }
$$
there exists a lift $L\otimes B\to  D$.
\end{rem}

\subsection{Compact metrizable spaces}
In this subsection we consider an application of Theorems \ref{t:main3} and \ref{t:special}. Let $\CM$ denote the category of compact metrizable spaces and continuous maps. Note that the category $\CM$ is naturally tensored over $\cS_f$ by the following action:
$$K\otimes X:=|K|\times X,$$
for $K\in\cS_f$ and $X\in\CM$ ($|-|$ denotes geometric realization).
Let
$$Y:A\mapsto \Map_{\CM}(A,-):\CM\to (\cS^{\CM})^{\op}$$
denote the enriched dual Yoneda embedding.

\begin{define}\label{d:CM WCC}
Consider the projective model structure on $\cS^{\CM}$. We define a map $f$ in $\CM$ to be a weak equivalence or a cofibration if $Y(f)$ is a weak equivalence or a cofibration in $(\cS^{\CM})^{\op}$.
\end{define}

\begin{prop}\label{p:CM is special}
With the weak equivalences and cofibrations defined in Definition \ref{d:CM WCC}, $\CM$ is a special weak cofibration category.
\end{prop}

\begin{proof}
By the Gel'fand--Na{\u{\i}}mark correspondence, the category of compact Hausdorff spaces and continuous maps between them is equivalent to the opposite category of commutative unital $C^*$-algebras and unital $*$-homomorphisms between them. Restricting to separable $C^*$-algebras, we obtain an equivalence between our category $\CM$ of compact metrizable spaces and the opposite category of commutative unital separable $C^*$-algebras. Since every unital separable $C^*$-algebra is isomorphic to a unital sub-$C^*$-algebra of the $C^*$-algebra of bounded operators on $l^2$, we see that the category of unital separable $C^*$-algebras is essentially small. In particular, we get that the category $\CM$ is essentially small. We can therefore assume that we are working with an equivalent small category, and we will do so without further mentioning.

It is not hard to see that the category $\CM$ has finite limits and colimits and the left action of $\cS_f$ on $\CM$ defined above is weakly closed.

Let $\Top$ denote the category of topological spaces. For $A,B\in \Top$, we denote by $\Hom_{\Top}(A,B)$ the set of continuous maps from $A$ to $B$ endowed with the compact open topology. Then, if $B$ is locally compact Hausdorff, we have a natural bijection
$$\Top(A,\Hom_{\Top}(B,C))\cong \Top(A\times B,C).$$
In particular, we see that for every $A,B\in \CM$ we have
$$\Sing(\Hom_{\Top}(A,B))_n\cong \Top(|\Delta^n|,\Hom_{\Top}(A,B))\cong \CM(\Delta^n\otimes A,B).$$
It follows that $\Sing(\Hom_{\Top}(A,B))\cong\Map_{\CM}(A,B)$ and thus the simplicial set $\Map_{\CM}(A,B)$ is a Kan complex. Now the proposition follows from Theorem \ref{t:special}.
\end{proof}

Our objective now is to give a more familiar description to the weak equivalences and the cofibrations in $\CM$.

\begin{define}\label{d:Hurewicz cofibration}
A map $i: X\to Y$ in $\CM$ is called a \emph{Hurewicz cofibration} if for every $Z\in\CM$ and every diagram of the form
$$
\xymatrix{ \{0\}\times Y\coprod_{\{0\}\times X}[0,1]\times X\ar[d]\ar[r] & Z\\
 [0,1]\times Y & }
$$
there exists a lift $[0,1]\times Y\to  Z$.
\end{define}

\begin{rem}\label{r:left inverse}
By taking $Z$ to be $\{0\}\times Y\coprod_{\{0\}\times X}[0,1]\times X$ and the horizontal arrow to be the identity in the definition above, it is not hard to see that a map $X\to Y$ is a Hurewicz cofibration iff the induced map
$$\{0\}\times Y\coprod_{\{0\}\times X}[0,1]\times X\to[0,1]\times Y$$
admits a left inverse (that is, a retracting map).

Since the cartesian product in $\CM$ commutes with finite colimits, if follows that for every Hurewicz cofibration $X\to Y$ in $\CM$ and every $A\in \CM$, the map  that $A\times X\to A\times Y$ is also a Hurewicz cofibration
\end{rem}

\begin{rem}\label{r:equiv Hurewicz}
By the exponential law for $\Top$ considered in the proof of Proposition \ref{p:CM is special}, it follows that a map $i: X\to Y$ in $\CM$ is a Hurewicz cofibration iff for every $Z\in\CM$ and every commutative diagram of the form
$$
\xymatrix{
\{0\} \ar[d] \ar[r] & \Hom_{\Top}(Y,Z)\ar[d]^{i^*}\\
[0,1]\ar[r] & \Hom_{\Top}(X,Z),
}$$
there exists a lift $[0,1]\to \Hom_{\Top}(Y,Z)$.
\end{rem}

\begin{lem}\label{equiv Cof}
Let $i: X\to Y$ be a map in $\CM$. Then the map $i$ is a Hurewicz cofibration iff for every $Z\in\CM$ the induced map $\Hom_{\Top}(Y,Z)\to \Hom_{\Top}(X,Z)$ is a Serre fibration.
\end{lem}

\begin{proof}
If for every $Z\in\CM$ the induced map $\Hom_{\Top}(Y,Z)\to \Hom_{\Top}(X,Z)$ is a Serre fibration, then clearly $i$ is a Hurewicz cofibration (see Remark \ref{r:equiv Hurewicz}).

Suppose $i$ is a Hurewicz cofibration. Let $Z$ be an object in $\CM$. It is enough to show that for every $n\geq 0$ and every commutative diagram of the form
$$
\xymatrix{
{\{0\}}\times [0,1]^{n} \ar[d] \ar[r] & \Hom_{\Top}(Y,Z)\ar[d]^{i^*}\\
[0,1]\times [0,1]^{n}\ar[r] & \Hom_{\Top}(X,Z),
}$$
there exists a lift $[0,1]^{n+1}\to \Hom_{\Top}(Y,Z)$.
By the exponential law for $\Top$ considered in the proof of Proposition \ref{p:CM is special}, this is equivalent to requiring that for every $n\geq 0$ and every diagram of the form
$$
\xymatrix{ \{0\}\times [0,1]^n\times Y\coprod_{\{0\}\times [0,1]^n\times X}[0,1]\times [0,1]^n\times X\ar[d]\ar[r] & Z\\
 [0,1]\times [0,1]^n\times Y & }
$$
there exists a lift $[0,1]^{n+1}\times Y\to Z$.
Again by the same rule, this is equivalent to requiring that for every $n\geq 0$ and every commutative diagram of the form

$$
\xymatrix{
{\{0\}} \ar[d] \ar[r] & \Hom_{\Top}( [0,1]^{n}\times Y,Z)\ar[d]\\
[0,1]\ar[r] & \Hom_{\Top}([0,1]^{n}\times X,Z),
}$$
there exists a lift $[0,1]\to \Hom_{\Top}( [0,1]^{n}\times Y,Z)$. But since $i:X\to Y$ is a Hurewicz cofibration, we have by Remark \ref{r:left inverse} that $[0,1]^{n}\times X\to [0,1]^{n}\times Y$ is also a Hurewicz cofibration, so the lift exists by Remark \ref{r:equiv Hurewicz}.
\end{proof}

\begin{prop}\label{p:Hurewicz}
Let $i: X\to Y$ be a map in $\CM$. Then the following hold:
\begin{enumerate}
\item The map $i$ is a weak equivalence in $\CM$ iff it is a homotopy equivalence.
\item The map $i$ is a cofibration in $\CM$ iff it is a Hurewicz cofibration.
\end{enumerate}
\end{prop}

\begin{proof}\
\begin{enumerate}
\item This follows from Definition \ref{d:CM WCC} and \cite[Proposition 1.2.4.1]{Lur}.
\item By Remark \ref{r:equiv cof}, we see that a map $i: X\to Y$ in $\CM$ is a cofibration iff for every $Z\in\CM$ and every acyclic cofibration $K\to L$ in $\cS_f$ and every diagram of the form

$$
\xymatrix{ |K|\times Y\coprod_{|K|\times X}|L|\times X\ar[d]\ar[r] & Z\\
 |L|\times Y & }
$$
there exists a lift $|L|\times Y\to  Z$. It follows that every cofibration in $\CM$ is a Hurewicz cofibration. The opposite direction follows from Lemma \ref{equiv Cof} and the fact that the functor
$\Sing:\Top\to\cS$
sends Serre fibrations to Kan fibrations.
\end{enumerate}
\end{proof}

\begin{prop}\label{p:CM_monoidal}
The weak cofibration category $\CM$ described in Proposition \ref{p:CM is special} is a monoidal weak cofibration category, under the categorical product.
\end{prop}

\begin{proof}
By taking $Z$ to be $\{0\}\times Y\coprod_{\{0\}\times X}[0,1]\times X$ and the horizontal arrow to be the identity in Definition \ref{d:Hurewicz cofibration}, it is not hard to see that every Hurewicz cofibration $X\to Y$ is an embedding. Since $X$ is compact and $Y$ is Hausdorff, it is also a closed embedding.
Now the result follows from Proposition \ref{p:Hurewicz}, \cite[Theorem 6]{Str} and \cite[Theorem 5.9]{Whi}.
\end{proof}

We obtain the following:

\begin{thm}\label{t:main4}
Let $\cW$ denote the class of weak equivalences and let $\cC$ denote the class of cofibrations in $\CM$. Then there exists a simplicial model category structure on $\Ind(\CM)$ such that:
\begin{enumerate}
\item The weak equivalences are $\mathbf{W} := \Lw^{\cong}(\cW)$.
\item The cofibrations are $\mathbf{C} := \R(\coSp^{\cong}(\cC))$.
\item The fibrations are $\mathbf{F} :=(\cC\cap \cW)^{\perp}$.
\end{enumerate}

Moreover, this model category is finitely combinatorial, with set of generating cofibrations $\cC$ and set of generating acyclic cofibrations $\cC\cap \cW$.

The model category $\Ind(\CM)$ has the following further properties:
\begin{enumerate}
\item Every object in $\Ind(\CM)$ is fibrant.
\item $\Ind(\CM)$ is proper.
\item $\Ind(\CM)$ is a cartesian monoidal model category.
\end{enumerate}
\end{thm}

\begin{proof}
By Proposition \ref{p:CM is special} $\CM$ is a special weak cofibration category, so the result is given almost entirely by Theorem \ref{t:main3}. We only need to show Property (3) above.

We have shown in  \cite{BaSc3} that the
natural prolongation of the categorical product in $\CM$ to a bifunctor
$$\Ind(\CM)\times \Ind(\CM)\to \Ind(\CM)$$
is exactly the categorical product in $\Ind(\CM)$.
Now the result follows from Proposition \ref{p:CM_monoidal} and Proposition \ref{p:monoidal}.
\end{proof}


Department of Mathematics, University of Muenster, Nordrhein-Westfalen, Germany.

\emph{E-mail address}:
\texttt{ilanbarnea770@gmail.com}

\end{document}